\documentclass[12pt,bezier]{article}
\usepackage{indentfirst}
\usepackage{changepage}
\usepackage{amssymb}
\usepackage{enumerate}
\usepackage{mathrsfs}
\usepackage{amsmath}
\usepackage{amsfonts,amsthm,amssymb}
\usepackage{amsfonts}
\usepackage{graphicx}
\usepackage{caption}
\usepackage{float}
\usepackage[numbers,sort&compress]{natbib}

\newtheorem{lem}{Lemma}[section]
\newtheorem{thm}[lem]{Theorem}

\newtheorem{defi}[lem]{Definition}
\newtheorem{con}[lem]{Conjecture}

\begin{document}
	
	\title{The minimum degree of minimally $ t $-tough graphs}
	
	\author{Hui Ma, Xiaomin Hu, Weihua Yang\footnote{Corresponding author. E-mail: ywh222@163.com; yangweihua@tyut.edu.cn}\\
		\\ \small Department of Mathematics, Taiyuan University of
		Technology,\\
		\small  Taiyuan Shanxi-030024,
		China\\
	}
	\date{}
	\maketitle
	
	{\small{\bf Abstract:} A graph $ G $ is minimally $ t $-tough if the toughness of $ G $ is $ t $ and deletion of any edge from $ G $ decreases its toughness. Katona et al. conjectured that the minimum degree of any minimally $ t $-tough graph is $ \lceil 2t\rceil $ and gave some upper bounds on the minimum degree of the minimally $ t $-tough graphs in \cite{Katona, Gyula}. In this paper, we show that a minimally 1-tough graph $ G $ with girth $ g\geq 5 $ has  minimum degree at most $  \lfloor\frac{n}{g+1}\rfloor+g-1$, and  a minimally $ 1 $-tough graph with girth $ 4 $ has  minimum degree at most $ \frac{n+6}{4}$. We also prove that the minimum degree of minimally $\frac{3}2$-tough claw-free graphs is $ 3 $.

		
		\section{Introduction}
		
		All graphs considered in this paper are finite, simple and undirected. Let $ G=\big(V(G), E(G)\big) $ be a graph. For any vertex $ v $ of $ G $, we denote the set of vertices of $ G $ adjacent to $ v $ by $ N_{G}(v) $ and the number of edges of $ G $ incident with $ v $ by $ d_{G}(v) $. Setting $ \delta(G)=min\left\{d_{G}(v) : v\in V(G)\right\} $. Let $ \kappa(G) $ denote the vertex connectivity of $ G $. For any vertex set $ S\subseteq V(G) $, we use $ G[S] $ and $ G-S $ to denote the subgraph of $ G $ induced by $ S $ and $ V(G)-S $, respectively,  and $ w(G-S) $ to denote the number of components of $ G-S $. For any vertex sets $ S, T\subseteq V(G) $, let $ N_{S}(T) $ denote the set of vertices in $ S-T $ adjacent to $ T $. A subset $ S $ is said to be independent if no two vertices of $ S $ are adjacent in $ G $.		~~~~~
		%\begin{defi}\label{defi3}
		%Let $x=(x_1,x_2,\cdots,x_n)$, $y=(y_1,y_2,\cdots,y_n)$, $x\bigcap y=\{i|x_i=y_i=1,1\leq i\leq n\}$.
		%\end{defi}
		\begin{defi}
			The notion of toughness was introduced by Chv\'atal (\cite{Chvatal}):
			$$ \tau(G)=min\left\{\dfrac{\lvert S\rvert}{w(G-S)} : S\subseteq V(G), w(G-S)\geq 2\right\}$$
			if $ G $ is not complete, and $ \tau(G)=+\infty $ if $ G $ is complete.
			We say that a graph is $ t-tough $, if its toughness is at least $ t $.
		\end{defi}

Since Chv\'atal  introduced toughness in 1973, many results concerning the relation  between toughness with cycle structure and factors have been widely studied (see \cite{Lu, Kab}). In the progress of research on the Hamiltonicity of $ 2 $-tough graph, Broersma introduced the concept of minimally $ t $-tough graphs (\cite{Broersma}).
		\begin{defi}
			Graph $ G $ is said to be minimally $ t $-tough, if $ \tau(G)=t $ and $ \tau(G-e)<t $ for any $ e\in E(G) $.
		\end{defi}
 From the definition of toughness, we note that a $t$-tough noncomplete graph is $2t$-connected. Thus the minimum degree of a $t$-tough graph is at least $ \lceil 2t\rceil$. It is well known that the minimum degree of every minimally $k$-connected graph is exactly $k$ (\cite{Mader}). Thus, the following is a natural question.

		\begin{con}[Kriesell \cite{Kaiser}]
			 Every minimally 1-tough graph has a vertex of degree 2.
		\end{con}
		
		\begin{con}[Generalized Krisell Conjecture \cite{Katona}]
		Every minimally t-tough graph has a vertex of degree $ \lceil 2t\rceil $.
	\end{con}	

Katona et al. considered Kriesell's conjecture and gave an interesting upper bound. We use $ n $ to denote the order of a graph.
		
		\begin{thm}[\cite{Gyula}]
			Every minimally 1-tough graph has a vertex of degree at most $ \frac{n}{3}+1 $.
		\end{thm}	

The girth of $ G $, denoted by $ g(G) $(or shortly, $g$), is the length of a shortest cycle in $ G $. The girth of the minimally 1-tough graph $G$ is useful in the work above. 
We shall prove the following result for $g\geq 5$ in Section 2.
		
\begin{thm}\label{result1}
		Every minimally 1-tough graph with girth $ g\geq 5 $ has a vertex of degree at most $ \lfloor\frac{n}{g+1}\rfloor+g-1 $.
		\end{thm}

A graph is said to be claw-free, if it does not contain an  induced subgraph isomorphic to $ K_{1,3} $. There are two open conjectures in terms of connectivity and toughness of claw-free graphs. Matthews and  Sumner conjectured that a 2-tough claw-free graph is hamiltonian (\cite{Matthews}), and Thomassen conjectured that a 4-connected line graph is hamiltonian (\cite{Thomassen}). The two conjectures are equivalent. In fact, a $2t$-connected claw-free graph is $t$-tough. Katona et al. proved that the minimum degree of minimally $ \frac{1}2 $-tough and $ 1 $-tough claw-free graphs are completely characterized in \cite{Gyula,Katona}.  Note that the toughness of a claw-free graph is exactly $\frac{c}2$ for some integer $c$. In this paper, we show that Generalized Krisell Conjecture is true if the graph is minimally $\frac{3}2$-tough claw-free. The proof is in Section 3.

 \begin{thm}\label{result2}
 Every minimally $\frac{3}2$-tough claw-free graph has a vertex of degree 3.
 \end{thm}

		%\begin{thm}\label{thm2.1}\cite{Godsil}
		%Let $G(V,E)$ be a vertex transitive graph, then $\alpha(G)\omega(G)\leq |V(G)|$.
		%\end{thm}
		
		%\begin{thm}\label{thm2.2}$[10]$
		%$Q_{n}^{(d,w)}$ is.
		%\end{thm}
		
		\section{Proof of Theorem \ref{result1}}
		
		\begin{lem}[\cite{Gyula}]\label{lem2.1}
		  If $ G $ is a minimally 1-tough graph, then for every edge $ e\in E(G) $ there exists a vertex set $ S\subseteq V(G) $ with	$ w(G-S)=\lvert S\rvert $ and $w\big(\left(G-e\right)-S\big)=\lvert S\rvert +1 $.
		\end{lem}
		
	The length of a shortest $ (u, v)- $path in $ G $ is called the distance between $ u $ and $ v $ in $ G $ and denoted by $ d_{G}(u, v) $.
	
\noindent{\bf Proof of Theorem 1.6.}
			Let $ G $ be a minimally 1-tough graph with order $ n $ and girth $ g\geq 5 $. Let $ C^{'} $ be a shortest cycle in $ G $ and $V(C^{'})=\left\{v_{i}\vert i=1, 2, \cdots, g\right\}$. Clearly, $ \big| N_{C^{'}}(v_{i})\big|=2 $ and $ \big(N_{G-C^{'}}\left(v_{i}\right)\big)\cap N_{G}(v_{j})=\emptyset $ for $ i\neq j $, otherwise we can obtain a shorter cycle, a contradiction. Thus $ n\geq g+g\big(\delta(G)-2\big) $. If $ g\geq \sqrt{n}-1 $, then $ \delta(G)\leq \frac{n}{g}+1\leq \lfloor\frac{n}{g+1}\rfloor+g-1 $.
		 Hence assume $ 5\leq g<\sqrt{n}-1 $ in the following.
			
			Let $ e=uv $ be an arbitrary edge of $G$. By Lemma~\ref{lem2.1}, there exists a vertex set $ S\subseteq V(G)$ such that $ w(G-S)=\lvert S\rvert $ and $ w\big(\left(G-e\right)-S\big)=\lvert S\rvert +1 $. Set $ \lvert S\rvert=k $. Since $ \big|V(G)\big|\geq \lvert S\rvert + w\big((G-e)-S\big)$, we have $ k<\frac{n}{2} $.
			
			\par
			{\bf Case 1.}
			 $ 1\leq k\leq g-1 $.
			 \par
			  Suppose to the contrary that $ \delta(G)\geq \lfloor\frac{n}{g+1}\rfloor+g $. Let $ C $ be the component of $ G-S $ containing $ e $ and $ D $ be the union of components of $  G-S-C $. Let $ C_{u} $ and $ C_{v} $ denote the components of $ (G-e)-S $ containing $ u $, $ v $, respectively. Assume $ V_{i}=\left\{x: d_{C_{v}}(x, v)=i\right\} $ for each nonnegative integer $ i $. 
			  
			   If $ g $ is odd, according to the definition of girth, then $ V_{i} $ is independent for each $ i\in\left[0, \frac{g-3}{2}\right] $. In addition, $ N_{_{V_{\frac{g-1}{2}}}}(u_{\frac{g-3}{2}})\cap N_{_{V_{\frac{g-1}{2}}}}(v_{\frac{g-3}{2}})=\emptyset $ for each $ u_{\frac{g-3}{2}}, v_{\frac{g-3}{2}}\in V_{\frac{g-3}{2}} $. Then we have the following:
			  	$$ \lvert C_{v}\rvert\geq \sum_{i=0}^{\frac{ g-1}{2}}\lvert V_{i}\rvert 
			  	\geq\sum_{i=0}^{\frac{g-1}{2}}\big(\delta\left(G\right)-k-1\big)^{i}. $$
	By a similar argument, we can obtain $$ \lvert C_{u}\rvert
	\geq \sum_{i=0}^{\frac{ g-1}{2}}\big(\delta\left(G\right)-k-1\big)^{i}.\nonumber $$
			Since $ \sqrt{n}-1>g\geq 5$, then we have $ \frac{\sqrt{n}-2}{2}>2$. Combining $ k\leq g-1 $ and $ g<\sqrt{n}-1 $, then,
		\begin{equation}
		\begin{aligned}
			n\geq \lvert C\rvert+\lvert S\rvert+\lvert D\rvert
			&\geq2\sum\limits_{i=0}^{\frac{ g-1}{2}}\big(\delta\left(G\right)-k-1\big)^{i}+2k-1
			\\&\geq 2\sum\limits_{i=0}^{\frac{ g-1}{2}}\big(\lfloor\frac{n}{g+1}\rfloor+g -(g-1)-1\big)^{i}+2(g-1)-1\\&>\sqrt{n}^{\frac{\sqrt{n}-2}{2}}+2\sqrt{n}-5>n, \nonumber		
		\end{aligned}
	\end{equation}
		a contradiction.

If $g$ is even, as $ \sqrt{n}-1>g\geq 6$, then $ \frac{\sqrt{n}-1}{2}-1>2 $. In the same conclusion as above, we have
		\begin{equation}
			\begin{aligned}
				n\geq \lvert C\rvert+\lvert S\rvert+\lvert D\rvert
				&\geq2\sum_{i=0}^{\frac{ g}{2}-1}\big(\delta\left(G\right)-k-1\big)^{i}+2k-1
		         >\sqrt{n}^{\frac{\sqrt{n}-1}{2}-1}+2\sqrt{n}-5>n, \nonumber		
			\end{aligned}
		\end{equation}
	a contradiction.
	\par
	{\bf Case 2.}
	 $ g\leq k\leq \lfloor\frac{n}{g+1}\rfloor $.
	 \par

 Since $ \lvert S\rvert=k $ and $ w\big((G-e)-S\big)=k+1 $, there exists a component of $ (G-e)-S $ with at most $ \lfloor\frac{n-k}{k+1}\rfloor $ vertices and inside this component each vertex has at most $ \lfloor\frac{n-k}{k+1}\rfloor-1 $ neighbors. If this component is trivial, then $ \delta(G)\leq \lfloor\frac{n}{g+1}\rfloor+1 $. Otherwise there exists a vertex which is not an endpoint of $e$ in this component, so its degree in $G$ is at most	
$ \lfloor\frac{n-k}{k+1}\rfloor -1+k\leq \frac{n-k}{k+1}-1+k=\frac{n+1}{k+1}+(k+1)-3 $.  

 	Consider the function
 \begin{center}
 	$ f(k)=\dfrac{n+1}{k+1}+(k+1)-3 $.
 \end{center}
Since $ g\leq k\leq \lfloor\frac{n}{g+1}\rfloor $, we see that $ \delta(G)\leq \lfloor f(k)\rfloor\leq max\left\{f(g), f(\lfloor\frac{n}{g+1}\rfloor)\right\}$. As $ 0<\frac{g+1}{n}<1 $ and
\begin{equation}
	\begin{aligned}
		f(g)&=\frac{n+1}{g+1}+g-2<\lfloor\frac{n}{g+1}\rfloor+g-1,\\
		f(\lfloor\frac{n}{g+1}\rfloor)&=\frac{n+1}{\lfloor \frac{n}{g+1}\rfloor +1}+\lfloor\frac{n}{g+1}\rfloor-2
		\leq \lfloor\frac{n}{g+1}\rfloor +g-1+\frac{g+1}{n},
		\nonumber
	\end{aligned}
\end{equation}
the result follows.

	{\bf Case 3.} $ \lfloor\frac{n}{g+1}\rfloor+1\leq k<\dfrac{n}{2} $.
	\par
 Suppose to the contrary $ \delta(G)\geq \lfloor\frac{n}{g+1}\rfloor+g $.
Let $ T_{1}, T_{2}, \cdots, T_{l} $ be the components of $ G-S $ with at most $ g-1 $ vertices. Then we have $ l+g(k-l)+k\leq n $. Moreover, since $ k\geq\delta(G)-1 $ and $ g\geq 5 $, we can obtain  
\begin{equation}
	\begin{aligned}
		l\geq \frac{g+1}{g-1}k-\frac{n}{g-1}\geq \frac{g+1}{g-1}\big(\lfloor\frac{n}{g+1}\rfloor+g-1\big)-\frac{n}{g-1}>g-1\geq \lceil \frac{g+1}{2}\rceil 
	.\nonumber
	\end{aligned}
\end{equation} 

We claim $ G[T_{i}] $ is a path for each $ i\in \left[1, l\right] $. Clearly, $ G[T_{i}] $ is a tree for each $ i\in \left[1, l\right] $. If $ d_{T_{j}}(u)\geq 3 $ for some $ j\in \left[1, l\right] $ and $ u\in V(T_{j}) $, then  $ \big|S\cup\left\{u\right\}\big|<w(G-S\cup\left\{u\right\})$, which contradicts $ \tau(G)=1 $. Assume $ u_{i} $ is the starting vertex of $ T_{i} $ for each $ i\in \left[1, l\right] $ and $ W=\left\{u_{i} : i=1, 2, \cdots, \lceil \frac{g+1}{2}\rceil\right\} $. Then
	\begin{equation}
		\begin{aligned}
			\sum_{i=1}^{l}\big|N_{S}(u_i)\big|&\geq 
			\sum_{i=1}^{\lceil \frac{g+1}{2}\rceil}\big|N_{S}(u_i)\big|&\geq \lceil  \frac{g+1}{2}\rceil\big(\lfloor\frac{n}{g+1}\rfloor+g-1\big)>\frac{n}{2}+\lceil \frac{g+1}{2}\rceil(g-2).\nonumber
		\end{aligned}
	\end{equation}
Since $ \lvert S\rvert<\frac{n}{2} $, there exists a vertex $ u_{j}\in W $ such that $ \Big|N_{S}(u_{j})\cap N_{S}\big(W-\left\{u_{j}\right\}\big)\Big|> g-2 $. Since $ g-2>\big|W-\left\{u_{j}\right\}\big|$, then there exists a vertex of $ W-\left\{u_{j}\right\} $ adjacent to at least two vertices of $ N_{S}(u_{j}) $. So there exists a 4-cycle in $ G $, a contradiction.\qed

We have the following result for $ g=4 $.

\begin{thm}
	 Every minimally 1-tough graph with girth $ 4 $ has a vertex of degree at most $ \frac{n+6}{4} $.
\end{thm}

\begin{proof}
Suppose to the contrary that $ \delta(G)>\frac{n+6}{4}$. Let $ e=uv $ be an arbitrary edge of $ G $ and $ S $ be a vertex set guaranteed by Lemma 2.1 with $ \lvert S\rvert=k$. Assume $ C $ is the component of $ G-S $ containing $ e $.
	\par
	Since $ G $ is triangle-free, we see that $ \lvert T\rvert\geq2\delta(G)-2-k $, where $ T=N_{C}(u)\cup N_{C}(v)-\left\{u, v\right\} $ and $ T $ is independent.  Since $ G $ is 1-tough and $ S\cup (C-T) $ is a cut set of $ G $, we can obtain $ \lvert C-T\rvert\geq \lvert T\rvert-1$. Then
	\begin{center}
		$ \lvert C\rvert\geq 2\lvert T\rvert-1\geq 4\delta(G)-k-5\geq n-2k+2 $.
	\end{center}
	Therefore, there are only $ n-\big(\left(n-2k+2\right)+k\big)=k-2 $ vertices in the remaining $ k-1 $ components, which is a contradiction.
\end{proof}

\section{Proof of Theorem \ref{result2}}

\begin{thm}[\cite{Matthews}]\label{result 8}
	If $ G $ is a noncomplete claw-free graph, then $ 2\tau(G)=\kappa(G) $.
\end{thm}

\par
If G is a minimally $ \dfrac{3}{2} $-tough claw-free graph, then $ \kappa(G)=3 $ by Theorem \ref{result 8}. Moreover, if $ G-e $ is claw-free for an arbitrary edge $ e\in E(G) $, we can obtain a characterization of the endpoints of $ e $.  A vertex $ u $ of $ G $ is called critical if  $ G $ is $ k $-connected and $ G-u $ is not $ k $-connected. And if every vertex of a noncomplete graph $ G $ is critical, then $ G $ is said to be critically $ k $-connected.
\begin{lem}
	Let $ G $ be a minimally $ \dfrac{3}{2} $-tough claw-free graph. For an arbitrary edge $ e=uv\in E(G) $, if $ e $ is not contained in any triangle, then each endpoint of $ e $ is either critical or of degree 3.
	
\end{lem}

\begin{proof}
	Since $ e $ is not contained in any triangle and $ G $ is claw-free, we see that $ G-e $ is claw-free. By Theorem 3.1 and the definition of minimally $ t-$ tough graphs, $ \kappa(G-e)=2\tau(G-e)<3 $. If $ \tau(G-e)=\dfrac{1}{2} $, then $ \kappa(G-e)=1 $ and thus $ \kappa(G)\leq 2 $, a contradiction. Thus $ \tau(G-e)=1 $ and $ \kappa(G-e)=2 $. Let $ \left\{w_{1}, w_{2}\right\} $ be a 2-vertex-cut of $ G-e $. If $ d_{G}(u)=3 $, the lemma follows. If $ d_{G}(u)\geq 4 $, then $ \left\{u, w_{1}, w_{2}\right\} $ is a 3-vertex-cut of $ G $. It follows that $ u $ is critical.  Similarly, $ d_{G}(v)=3 $ or $ v $ is critical.
\end{proof}

From the definition of minimally $ t $-tough graphs, we can obtain the following lemma.
\begin{lem}
	If $ G $ is a minimally $ \dfrac{3}{2} $-tough graph, then for every edge $ e\in E(G) $ there exists a vertex set $ S\subseteq V(G) $ with
	\begin{center}
	 $ \dfrac{3}{2}{w(G-S)} \leq \lvert S\rvert< \dfrac{3}{2}{w\big(\left(G-e\right)-S\big)}= \dfrac{3}{2}{\big(w\left(G-S\right)+1\big)}$.
	\end{center}
\end{lem}

  For readability of the proofs of following lemmas and Theorem 1.8, let $ S(e) $ be a minimum vertex set guaranteed by Lemma 3.3 for every edge $ e=uv\in E(G) $ and $ \big|S(e)\big|=k(e) $. Let $ C(e) $ be the component of $ G-S(e) $ containing $ e $ and $ D(e) $ be the union of components of $ G-S(e)-C(e) $. In particular, let $ C_{u}(e)$ and $C_{v}(e) $ denote the components of $ \left(G-e\right)-S(e) $ containing $ u $ and $ v $, respectively. In addition,  Let $ N_{S}(T_{1}, T_{2}, \cdots, T_{l})=\bigcap_{i=1}^{l} N_{S}(T_{i}) $, where $S, T_{i}\subseteq V(G) $  for $ i=1, 2, \cdots, l $. In particular, if $ T_{i}=\left\{t_{i}\right\} $, we will abbreviate $ N_{S}(T_{1}, \cdots, \left\{t_{i}\right\}, \cdots, T_{l})$ to $ N_{S}(T_{1}, \cdots, t_{i}, \cdots, T_{l})$.
 
 The following lemma plays a key role in the proof of Theorem 1.7.
 
\begin{lem}
	Let $ G $ be a minimally $ \dfrac{3}{2} $-tough claw-free graph with $ \delta(G)\geq 4 $. For an arbitrary edge $ e=uv\in E(G) $, assume $ w\in S(e) $. The following statements hold.
	\begin{enumerate}[(1)]
		\item If $ k(e)=3 $, then $ C(e)=\left\{u, v\right\} $, $ N_{G}(u)=\left\{v\right\}\cup S(e)$ and $ N_{G}(v)=\left\{u\right\}\cup S(e) $.
		\item If $ k(e)\geq 5 $, then $ \Big|N_{ S(e)}\big(C(e)\big)\Big|= 5 $ and every vertex of $ S(e) $ is adjacent to $ D(e) $. \label{result4}
		\item If $ D(e)\neq \emptyset $ and $ w $ is not adjacent to $ D(e) $, then $ k(e)=4 $ and $N_{S(e)}\big(D(e)\big)= S(e)\backslash\left\{w\right\} $.\label{result5} 
		\item If $ w $ is not critical, then $ k(e) =4 $, $ C(e)=\left\{u, v\right\} $ and each vertex of $ S(e)\backslash\left\{w\right\} $ is adjacent to $ D(e) $.\label{result6}
		\item If $ u $ is not critical, $ D(e)\neq\emptyset $ and $ w\in N_{S(e)}\big(u, v, D(e)\big)$, then $ d_{G}(v)= 4 $,  $ 2\leq \big|N_{G}(v, w)\big|\leq 3 $ and $ N_{G}(w)\subseteq \left\{v, t\right\}\cup N_{G}(v)$, where $ t\in N_{D(e)}(w) $.
		\item  If $ u $ is not critical, then $ C_{u}(e)=\left\{u\right\} $.\label{result7}
	
	\end{enumerate}
	
\end{lem}

\begin{proof}
	(1). By Lemma 3.3, $ w\big(G-S(e)\big)=2 $ and  $ w\big(\left(G-e\right)-S(e)\big)=3 $.  
	Note that $ C_{u}(e)=\left\{u\right\} $, otherwise $ \dfrac{\big|S(e)\cup \left\{u\right\}\big|}{w\big(G-S(e)\cup \left\{u\right\}\big)}\leq\dfrac{4}{3}<\dfrac{3}{2} $, a contradiction. Similarly, $ C_{v}(e)=\left\{v\right\} $. Since $ \delta(G)\geq 4 $, then $ N_{G}(u)=\left\{v\right\}\cup S(e)$ and $ N_{G}(v)=\left\{u\right\}\cup S(e) $.
	
	 (2). Let $ w\big(G-S(e)\big)=l $ and $ C_{1}=C(e), C_{2}, \cdots, C_{l} $ be the components of $ G-S(e) $. By Lemma 3.3, $ \dfrac{3}{2}l\leq k(e)< \dfrac{3}{2}l+\dfrac{3}{2} $.
	 
	 Since $ G $ is 3-connected, we have $ \big|N_{S(e)}(C_{i})\big|\geq 3 $ for each $ i\in\left\{1, 2, \cdots, l \right\} $. Thus there are at least 3$ l $ edges coming from the components of $ G-S(e) $ to $ S(e) $ counting at most one from any component $ C_{i} $ to a particular vertex of $ S(e) $. Since $ G $ is claw-free, every vertex of $ S(e) $ have neighbors in at most two components of $ G-S(e) $. Then there are at most 2$ k(e) $ edges coming from $ S(e) $ to the components of $ G-S(e) $ counting at most one edge from every component to a particular vertex of $ S(e) $. Therefore we conclude $ 3l\leq 2k(e)<2(\dfrac{3}{2}l+\dfrac{3}{2}) $ and then $ 3\leq\Big|N_{S(e)}\big(C(e)\big)\Big|\leq 5 $. 
	 Suppose $3\leq \Big|N_{S(e)}\big(C(e)\big)\Big|\leq 4$. Obviously, $ N_{S(e)}\big(C(e)\big) $ is a vertex set guaranteed by Lemma 3.3, which contradicts the minimality of $ S(e) $. So $ \Big|N_{S(e)}\big(C(e)\big)\Big|=5 $ and then $ 5+3(l-1)\leq 2k(e)<2(\dfrac{3}{2}l+\dfrac{3}{2}) $. Hence every vertex of $ S(e) $ has neighbors in exactly two components of $ G-S(e) $. Then the statement follows.
	
	(3). Combining $ \kappa(G)=3 $ with Lemma 3.4 (2), we have $ k(e)=4 $ and $ N_{S(e)}\big(D(e)\big)=S(e)\backslash\left\{w\right\} $.  
	
    (4). By Lemma 3.4 (1) and (2), $k(e)\in\left\{2, 4\right\} $.
	If $ k(e)=2 $, we have $ w\big(G-S(e)\big)=1 $ by Lemma 3.2.  Suppose either $ \big|C_{u}(e)\big|\geq2 $ or $ \big|C_{v}(e)\big|\geq2 $, then either $ S(e)\cup \left\{u\right\} $ or $ S(e)\cup \left\{v\right\} $ is a 3-vertex-cut of $ G $, a contradiction. Suppose $ \big|C_{u}(e)\big| =\big| C_{v}(e)\big|=1 $, then $ G=K_{4} $, which contradicts the toughness of $ G $. Therefore $ k(e)=4 $ and then $ w\big(G-S(e)\big)=2 $ by Lemma 3.2.  
	\par
Suppose $ C_{u}(e)\neq\left\{u\right\} $. Since $ \left\{u, w\right\}\cup N_{S(e)-\left\{w\right\}}\big(C_{u}(e)\backslash\left\{u\right\}\big)$ is a cut set of $ G $ and $ w $ is not critical, we have $ \Big|N_{S(e)-\left\{w\right\}}\big(C_{u}(e)\backslash\left\{u\right\}\big)\Big|\geq 2 $. Let $ S(e)=\left\{w, t_{1}, t_{2}, t_{3}\right\} $. Assume $ t_{1} $ and $ t_{2} $ are adjacent to $ C_{u}(e)\backslash\left\{u\right\} $. Since $ w $ is not contained in any 3-vertex-cut of $ G $, we see that $ N_{D(e)}(t_{i})\neq \emptyset $ for each $ i\in \left\{1, 2, 3\right\} $. Since $ G $ is claw-free, then $ t_{1} $ and $ t_{2} $ are not adjacent to $ C_{v}(e) $. Obviously, $ N_{G}\big(C_{v}(e)\big)\subseteq\left\{u, w, t_{3}\right\}$ is a cut set of $ G $ (see Figure 1), a contradiction. So $ C_{u}(e)=\left\{u\right\} $. Similarly, we can obtain $ C_{v}(e)=\left\{v\right\} $.

(5). Let $ e_{1}=vw $. Clearly, $ u\in S(e_{1}) $. Since $ u $ is not critical, we can obtain $ k(e_{1})=4 $ and $ C(e_{1})=\left\{v, w\right\}$ by Lemma 3.4 (\ref{result6}). Let $ t\in N_{D(e)}(w) $. Clearly, $ t\in S(e_{1}) $. As $ vt\notin E(G) $ and $ \delta(G)\geq 4 $, then $ d_{G}(v)=4 $, $ S(e_{1})=\left\{t\right\}\cup \big(N_{G}(v)\backslash\left\{w\right\}\big) $ and $ 2\leq \big|N_{G}(v, w)\big|\leq 3 $. It follows that $ N_{G}(w)\subseteq \left\{v, t\right\}\cup N_{G}(v)$. 

(6). Suppose $ C_{u}(e)\neq\left\{u\right\} $. Since $ u $ is noncritical and $ \left\{u\right\}\cup N_{S(e)}\big(C_{u}(e)\backslash\left\{u\right\}\big) $ is a cut set of $ G $, then  $\Big|N_{S(e)}\big(C_{u}(e)\backslash\left\{u\right\}\big)\Big|\geq 3 $ . If $ k(e)=3 $, we can obtain $ C_{u}(e)=\left\{u\right\} $ by Lemma 3.4 (1). If $ k(e)=4 $, since $ \kappa(G)=3 $ and $ G $ is claw-free, then there are at least two vertices of $ N_{S(e)}\big(C_{u}(e)\backslash\left\{u\right\}\big) $ adjacent to $ D(e) $ and not adjacent to $ C_{v}(e) $ (see Figure 2). Thus $ \Big| N_{S(e)}\big(C_{v}(e)\big)\Big|\leq 2 $, while $ N_{G}\big(C_{v}(e)\big)=\left\{u\right\}\cup N_{S(e)}\big(C_{v}(e)\big) $ is a cut set of $ G $, a contradiction. So $ C_{u}(e)=\left\{u\right\} $. If $ k(e)\geq 5 $, we know every vertex of $ N_{S(e)}\big(C_{u}(e)\backslash\left\{u\right\}\big) $ is adjacent to $ D(e) $ by Lemma 3.4 (2), and not adjacent to $ C_{v}(e) $ because $ G $ is claw-free. In the same conclusion as above, we can obtain $ C_{u}(e)=\left\{u\right\} $.
\end{proof}
	\begin{figure}[htbp]
	\centering
	\begin{minipage}[t]{0.48\textwidth}
		\centering
		\includegraphics[scale=0.4]{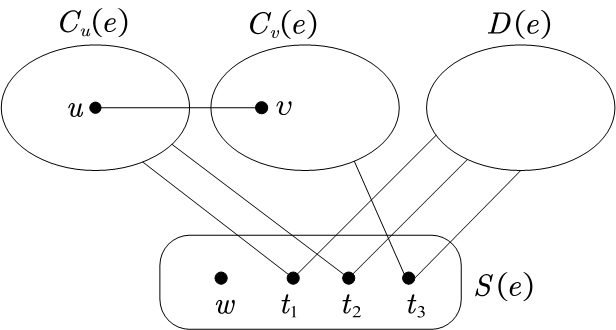}
		\caption{$ C_{u}\neq \left\{u\right\} $}
	\end{minipage}
	\begin{minipage}[t]{0.48\textwidth}
	\centering
	\includegraphics[scale=0.4]{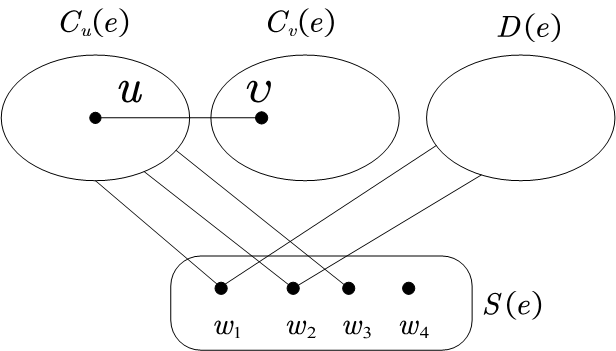}
	\caption{$ u $ is noncritical and $ C_{u}\neq\left\{u\right\} $}
\end{minipage}
\end{figure}

\begin{lem}
	Let $ G $ be a minimally $ \frac{3}{2} $-tough claw-free graph with $ \delta(G)\geq 4 $. For an arbitrary edge $ e=uv\in E(G) $, if $ u $ is not critical and $ k(e)\geq 4 $, then $ C(e)=\left\{u, v\right\} $ and $ \big|N_{S(e)}(u, v)\big|= 2 $.
\end{lem}

\begin{proof}
 By Lemma 3.2, $ uv $ must be contained in a triangle of $ G $. Assume $ w_{1}\in N_{G}(u, v) $ and $ S(e)=\left\{w_{i} | i=1, 2, \cdots, k(e)\right\}$. 	

 First, we will prove $ C(e)=\left\{u, v\right\} $. By Lemma 3.4 (6), $ C_{u}(e)=\left\{u\right\} $. Suppose to the contrary that $ C_{v}(e)\neq \left\{v\right\} $ and assume $ z\in N_{C_{v}(e)}(v) $. As $ \delta(G)\geq 4 $, then assume $ \left\{w_{2}, w_{3}\right\}\subseteq N_{G}(u) $. Consider the following two cases according to the common neighbors of $ u $ and $ C_{v}(e)\backslash\left\{v\right\} $.  
 
	{\bf Case 1.} 
		$ N_{S(e)}\big(u, C_{v}(e)\backslash\left\{v\right\}\big)= \emptyset $.
		
		As $ \kappa(G)=3 $, then $ \Big|N_{S(e)}\big(C_{v}(e)\backslash\left\{v\right\}\big)\Big|\geq 2 $. And as $ \big|N_{S(e)}(u)\big|\geq 3 $, then $ k(e)\geq 5 $. By Lemma 3.4 (2), we can assume $ N_{S(e)}\big(C(e)\big)=\left\{w_{1}, w_{2}, w_{3}, w_{4}, w_{5}\right\} $. Then $N_{S(e)}\big(C_{v}\left(e\right)\backslash\left\{v\right\}\big)=\left\{w_{4}, w_{5}\right\}$ and $ N_{G}(u)=\left\{v, w_{1}, w_{2}, w_{3}\right\} $.
		
		By Lemma 3.4 (2), $ w_{1}\in N_{S(e)}\big(D(e)\big) $. Since $ w_{1}\in N_{S(e)}(u, v, D(e)\big) $, then $\big|N_{G}(v, w_{1})\big|\geq 2$ by Lemma 3.4 (5). Obviously, $ N_{G}(v, w_{1})\subseteq \left\{u\right\}\cup C_{v}(e)\cup N_{S(e)}(C(e))$. Since $ w_{1}\in N_{S(e)}\big(u, D(e)\big) $ and $ G $ is claw-free, then $ N_{_{C_{v}(e)\backslash\left\{v\right\}}}(w_{1})=\emptyset$. Hence $ N_{G}(v, w_{1})\subseteq \left\{u, w_{2}, w_{3}, w_{4}, w_{5}\right\} $. 
		
		Suppose $ w_{2}\in N_{G}(v, w_{1}) $. By Lemma 3.4 (2), $\left\{w_{1}, w_{2}\right\}\subseteq N_{S(e)}\big(D(e)\big) $. As $ \left\{w_{1}, w_{2}\right\}\subseteq N_{S(e)}\big(u, v, D(e)\big) $, then $ N_{G}(v)=\left\{u, z, w_{1}, w_{2}\right\} $ and $ w_{3}\notin N_{G}(w_{1})\cup N_{G}(w_{2}) $ by  Lemma 3.4 (5). It follows that $ \left\{v, w_{1}, w_{2}\right\}\cap N_{G}(w_{3})=\emptyset $ and then $ uw_{3} $ is not contained in any triangle, which contradicts Lemma 3.2. Suppose $ w_{3}\in N_{G}(v, w_{1}) $. Similarly, we can obtain $ uw_{2} $ is not contained in any triangle, a contradiction.
		
		 Suppose $ w_{4}\in  N_{G}(v, w_{1}) $. By Lemma 3.4 (2), assume $ t_{1}\in N_{D(e)}(w_{1}) $. By Lemma 3.4 (5), then $ N_{G}(v)=\left\{u, z, w_{1}, w_{4}\right\} $ and $ N_{G}(w_{1})\subseteq\left\{v, t_{1}\right\}\cup N_{G}(v) $. As $ zw_{1}\notin E(G) $ and $ \delta(G)\geq 4 $, then $ N_{G}(w_{1})=\left\{u, v, w_{4}, t_{1}\right\}$. Assume $ t_{2}\in N_{_{C_{v}(e)\backslash\left\{v\right\}}}(w_{4}) $ and $ t_{3}\in N_{D(e)}(w_{4}) $ by Lemma 3.4 (2). We have $ t_{2}=z $ and $ t_{3}=t_{1} $, otherwise either $ G\big[\left\{w_{4}, t_{2}, v, t_{3}\right\}\big]=K_{1,3} $ or $ G\big[\left\{w_{4}, t_{2}, w_{1}, t_{3}\right\}\big]=K_{1,3} $, a contradiction. So $ \left\{t_{1}, z\right\}\subseteq N_{G}(w_{4}) $. Let $ e_{1}=w_{1}w_{4} $ (see Figure 3). Clearly, $ \left\{v, t_{1}\right\}\subseteq S(e_{1}) $.
		 
		{\bf Claim.}
	$k(e_{1})=4, N_{C(e_{1})}(t_{1})=\left\{w_{1}, w_{4}\right\} $, $ N_{C_{w_{1}}(e_{1})}(w_{1})=\left\{u\right\} $ and 	$ N_{C_{w_{1}(e_{1})}}(v)=\left\{u, w_{1}\right\}$.	 
		 
		  Since $ N_{G}(v)\subseteq N_{G}(w_{1})\cup N_{G}(w_{4}) $, then $ v\notin N_{S(e_{1})}\big(D(e_{1})\big) $. By Lemma 3.4 (3), $ k(e_{1})=4 $ and $ t_{1}\in N_{S(e_{1})}\big(D(e_{1})\big)$. Since $ t_{1}\in N_{S(e_{1})}\big(w_{1}, w_{4}, D(e_{1})\big) $ and $ G $ is claw-free, then $ N_{C(e_{1})}(t_{1})=\left\{w_{1}, w_{4}\right\} $. In addition, $ u\notin S(e_{1}) $. Otherwise, by Lemma 3.4 (3), $ S(e_{1})\backslash\left\{v\right\}=N_{S(e_{1})}\big(D(e_{1})\big) $ is a 3-vertex-cut  $ G $, which contains $ u $, a contradiction. So $ u\in C_{w_{1}}(e_{1}) $ and $ N_{C_{w_{1}}(e_{1})}(w_{1})=\left\{u\right\} $. Since $ zw_{4}\in E(G) $, then $ z\notin C_{w_{1}}(e_{1}) $ and $ N_{C_{w_{1}}(e_{1})}(v) =\left\{u, w_{1}\right\} $. \qed  
		  
		   If $ C_{w_{1}}(e_{1})\neq \left\{u, w_{1}\right\} $, then $N_{G}\big(C_{w_{1}}(e_{1})\backslash\left\{u, w_{1}\right\}\big) \subseteq\left\{u\right\}\cup \big(S(e_{1})\backslash\left\{v, t_{1}\right\}\big) $ is a cut set of $ G $ by Claim, a contradiction. If $ C_{w_{1}}(e_{1})=\left\{u, w_{1}\right\} $, then $ S(e_{1})= \left\{v, t_{1}, w_{2}, w_{3}\right\} $. As $ zw_{4}\in E(G) $, then $ z\in C_{w_{4}}(e_{1}) $ and $ C_{w_{4}}(e_{1})\neq \left\{w_{4}\right\} $. Moreover, since $ S(e_{1})\backslash\left\{v\right\}=N_{S(e)}\big(C_{w_{1}}(e_{1}), D(e_{1})\big) $ by Lemma 3.4 (3) and $ G $ is claw-free, then the vertices of $ S(e_{1})\backslash\left\{v\right\} $ are not adjacent to $ C_{w_{4}}(e_{1})\backslash\left\{w_{4}\right\} $. It follows that $N_{G}\big(C_{w_{4}}(e_{1})\backslash\left\{w_{4}\right\}\big)=\left\{v, w_{4}\right\} $ is a 2-vertex-cut of $ G $, a contradiction.  
		   
		   In the same conclusion as above, we can obtain $ w_{5}\notin  N_{G}(v, w_{1}) $. Therefore we conclude $ N_{G}(v, w_{1})=\left\{u\right\} $, a contradiction with Lemma 3.4 (5).
			\begin{figure}[htbp]
			\centering
			\begin{minipage}[t]{0.48\textwidth}
				\centering
				\includegraphics[scale=0.4]{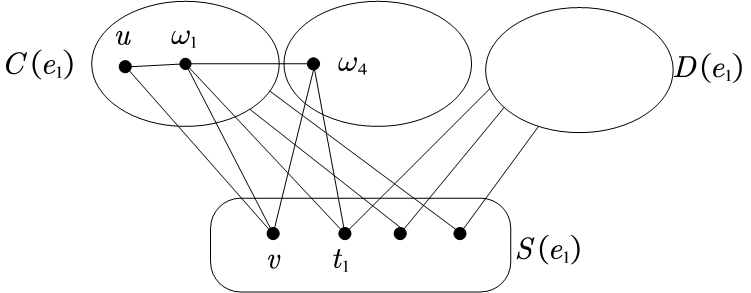}
				\caption{$ e_{1}=w_{1}w_{4} $}
			\end{minipage}
		\end{figure}
	
	{\bf Case 2.}
			$ \left\{w_{i}\right\}\subseteq N_{S(e)}\big(u, C_{v}(e)\backslash\left\{v\right\}\big)$ for some $ i\in \left\{1, 2, \cdots, k(e)\right\} $.
			
			Since $ G $ is claw-free, then $ w_{i}\notin N_{S(e)}\big(D(e)\big)$. By Lemma 3.4 (3), $ k(e)=4 $ and $ N_{S(e)}\big(D(e)\big)=S(e)\backslash\left\{w_{i}\right\} $. Thus $ N_{S(e)}\big(u, C_{v}(e)\backslash\left\{v\right\}\big)=\left\{w_{i}\right\} $. And as $ \kappa(G)=3 $, then $ \Big|N_{S(e)}\big(C_{v}(e)\backslash\left\{v\right\}\big)\Big|\geq 2 $. So there exists a vertex $ w_{j}\in N_{S(e)}\big(D(e), C_{v}(e)\backslash\left\{v\right\}\big) $ for some $ j\neq i $. Since $ G $ is claw-free, then $ uw_{j}\notin E(G) $. Thus $ w_{j}=w_{4} $ and $ N_{G}(u)=\left\{v, w_{1}, w_{2}, w_{3}\right\} $.
	
				{\bf Case 2.1.}
				$ N_{S(e)}\big(u, C_{v}(e)\backslash\left\{v\right\}\big)= \left\{w_{i}\right\}$ for some $ i\in \left\{2, 3\right\} $.
				
				Without loss of generality, we assume $ N_{S(e)}\big(u, C_{v}(e)\backslash\left\{v\right\}\big)=\left\{w_{2}\right\} $. Then $ N_{S(e)}\big(D(e)\big)=\left\{w_{1}, w_{3}, w_{4}\right\} $. By a similar argument as Case 1, we can obtain $\big|N_{G}(v, w_{1})\big|\geq 2$ and $ N_{G}(v, w_{1})\subseteq \left\{u, w_{2}, w_{3}, w_{4}\right\} $.
				
				 Consider that $ w_{2}\in N_{G}(v, w_{1}) $. As $ w_{1}\in N_{S(e)}\big(u, v, D(e)\big)$, then $ N_{G}(v)=\left\{u, z, w_{1}, w_{2}\right\} $ and $ w_{3}\notin N_{G}(w_{1}) $ by Lemma 3.4 (5). If $ w_{2}w_{3}\notin E(G) $, then $ uw_{3} $ is not contained in any triangle, which contradicts Lemma 3.2. If $ w_{2}w_{3}\in E(G) $, as $ G\big[\left\{v, w_{2}, w_{3}, w\right\}\big]=K_{1,3} $ for each $ w\in N_{_{C_{v}(e)\backslash\left\{v\right\}}}(w_{2})$ and $ w\neq z $, we see that $ N_{_{C_{v}(e)\backslash\left\{v\right\}}}(w_{2})=\left\{z\right\} $. Suppose $ C_{v}(e)\neq \left\{v, z\right\}$, then $N_{G}\big(C_{v}(e)\backslash\left\{v, z\right\}\big)\subseteq\left\{z, w_{4}\right\}$ is a cut set of $ G $ by our assumption, a contradiction. Suppose $ C_{v}(e)=\left\{v, z\right\} $, then $ N_{G}(z)\subseteq\left\{v, w_{2}, w_{4}\right\} $, a contradiction. Consider that $ w_{3}\in N_{G}(v, w_{1}) $. As $ \left\{w_{1}, w_{3}\right\}\subseteq N_{S(e)}\big(u, v, D(e)\big) $, then $ N_{G}(v)=\left\{u, z, w_{1}, w_{3}\right\} $ and $ w_{2}\notin N_{G}(w_{1})\cup N_{G}(w_{3}) $ by Lemma 3.4 (5). Therefore $ \left\{v, w_{1}, w_{3}\right\}\cap N_{G}(w_{2})=\emptyset $ and then $ uw_{2} $ is not contained in any triangle, which contradicts Lemma 3.2. By a similar argument as that in Case 1, we can obtain a contradiction for $ w_{4}\in N_{G}(v, w_{1}) $. Therefore we conclude $ N_{G}(v, w_{1})=\left\{u\right\} $, a contradiction.
				
				{\bf Case 2.2.}
					$ N_{S(e)}\big(u, C_{v}(e)\backslash\left\{v\right\}\big)=\left\{w_{1}\right\} $.
					
					 We claim that either $ vw_{2}\in E(G) $ or $ vw_{3}\in E(G) $. Otherwise, since $ \left\{w_{2}, w_{3}\right\}\cap N_{G}\big(C_{v}(e)\big)=\emptyset$, then $ N_{G}\big(C_{v}(e)\big)=\left\{u, w_{1}, w_{4}\right\} $ is a 3-vertex-cut of $ G $, a contradiction. Without loss of generality, assume $ vw_{2}\in E(G) $. As $ w_{2}\in N_{S(e)}\big(u, v, D(e)\big) $, then $ N_{G}(v)=\left\{u, z, w_{1}, w_{2}\right\} $ and $  N_{G}(w_{2})\subseteq\left\{v, y_{1}\right\}\cup N_{G}(v) $ by Lemma 3.4 (5), where $ y_{1}\in N_{D(e)}(w_{2}) $. Moreover, since $ zw_{2}\notin E(G) $ and $ d_{G}(w_{2})\geq 4 $, then $ w_{1}w_{2}\in E(G) $. Let $ e_{2}=w_{1}w_{2} $. Clearly, $ u\in S(e_{2}) $. By Lemma 3.4 (4), $ C(e_{2})=\left\{w_{1}, w_{2}\right\} $ and $ k(e_{2})=4 $. This means $ S(e_{2})=\left\{u, v, y_{1}, y_{2}\right\} $ and $ N_{G}(w_{1})\subseteq  S(e_{2})\cup \left\{w_{2}\right\} $, where $ y_{2}\in N_{_{C_{v}(e)\backslash\left\{v\right\}}}(w_{1}) $. Therefore $ \left\{v, w_{1}, w_{2}\right\}\cap N_{G}(w_{3})=\emptyset $ and then $ uw_{3} $ is not contained in any triangle, which contradicts Lemma 3.2. 
	
	Therefore $ C(e)=\left\{u, v\right\} $. Next, we will prove that $ \big|N_{S(e)}(u, v)\big|=2 $.
	 
	 Suppose $ \big|N_{S(e)}(u, v)\big| \geq 3 $. Assume $\left\{w_{1}, w_{2}, w_{3}\right\}\subseteq N_{S(e)}(u, v) $. Combing $ \kappa(G)=3 $ and Lemma 3.4 (2), we can assume $ w_{1}\in N_{S(e)}\big(D(e)\big) $. By Lemma 3.4 (5), $ N_{G}(v)=\left\{u, w_{1}, w_{2}, w_{3} \right\} $. Hence $ N_{G}(u)=\left\{v\right\}\cup N_{S(e)}\big(C(e)\big) $. If $ k(e)=4 $, then $ N_{S(e)}\big(C(e)\big)=S(e) $ by the minimality  of $ S(e) $. If $ k(e)\geq 5 $, then $ \Big| N_{S(e)}\big(C(e)\big)\Big|=5 $ by Lemma 3.4 (2). So $ d_{G}(u)\geq 5 $. Let $ e_{3}=uw_{1}$. Clearly, $ v\in S(e_{3}) $. As $ N_{G}(v)\subseteq\left\{u\right\}\cup N_{G}(u) $, then $ N_{D(e_{3})}(v)= \emptyset $. By Lemma 3.4 (\ref{result5}), $ k(e_{3})=4 $ and $ N_{S(e_{3})}\big(D(e_{3})\big)=S(e_{3})\backslash\left\{v\right\} $. Moreover, since $ C_{u}(e_{3})=\left\{u\right\} $ by Lemma 3.4 (\ref{result7}) and $ d_{G}(u)\geq 5 $, then $ S(e_{3})= N_{G}(u)\backslash\left\{w_{1}\right\} $. It follows that $ N_{D(e)}(w_{1})\subseteq C_{w_{1}}(e_{3}) $ and  $ S(e_{3})\backslash\left\{v\right\}\subseteq N_{S(e_{3})}\big(u, D(e_{3})\big) $. Since $ G $ is claw-free, then the vertices in $ S(e_{3})\backslash\left\{v\right\} $ are not adjacent to $ C_{w_{1}}(e_{3})\backslash\left\{w_{1}\right\}$. And as $ N_{G}(v)\subseteq\left\{u, w_{1}\right\}\cup S(e_{3}) $, then $ N_{G}\big(C_{w_{1}}(e_{3})\backslash\left\{w_{1}\right\}\big)=\left\{w_{1}\right\} $ is a cut set of $ G $, a contradiction. Thus $ \big|N_{S(e)}(u, v)\big|\leq 2 $. 
	
	If $ k(e)=4 $, as $ \delta(G)\geq 4 $, then $ \big|N_{S(e)}(u, v)\big|=2 $. If $ k(e)\geq 5 $, then assume $ N_{S(e)}\big(C(e)\big)=\left\{w_{1}, w_{2}, w_{3}, w_{4}, w_{5}\right\} $ by Lemma 3.4 (2). As $ \delta (G)\geq 4 $, then $ \big|N_{S(e)}(u, v)\big|\geq1 $. Suppose $ \big|N_{S(e)}(u, v)\big|\\=1$. As $ w_{1}\in N_{S(e)}(u, v) $, then $N_{S(e)}(u, v)=\left\{w_{1}\right\} $. Assume $ N_{G}(u)=\left\{v, w_{1}, w_{2}, w_{3}\right\} $ and $ N_{G}(v)=\left\{u, w_{1}, w_{4}, w_{5}\right\} $. By Lemma 3.4 (2), assume $ z_{1}\in N_{D(e)}(w_{1}) $. By Lemma 3.4 (5), $ N_{G}(w_{1})\subseteq\left\{v, z_{1}\right\}\cup N_{G}(v) $. Since $ d_{G}(w_{1})\geq 4 $, then either $ w_{1}w_{4}\in E(G) $ or $ w_{1}w_{5}\in E(G) $. Consider that $ w_{1}w_{4}\in E(G) $. Note that $ w_{4}w_{5}\in E(G) $, otherwise  $ G\big[\left\{u, v, w_{4}, w_{5}\right\}\big]=K_{1,3} $, a contradiction. Thus $ N_{G}(v)\subseteq N_{G}(w_{1})\cup N_{G}(w_{4}) $. And we have $ z_{1}\in N_{G}(w_{4}) $, otherwise $ G\big[\left\{u, w_{1}, w_{4}, z_{1}\right\}\big]=K_{1,3} $, a contradiction. Similarly as in Case 1, one can derive a contradiction for considering the edge $ w_{1}w_{4} $ as well. Similarly, we can also obtain a contradiction for  $ w_{1}w_{5}\notin E(G) $. Thus we can obtain $ \big|N_{S(e)}\left(u, v\right)\big|=2 $.
\end{proof}

\begin{thm}[\cite{Entringer}]
	In any critically 3-connected graph there are at least two vertices of degree 3.
\end{thm}

\noindent{\bf Proof of Theorem 1.7.}
If $ G $ is critically 3-connected, then $ \delta(G)=3 $ by Theorem 3.6. In the following assume that $ G $ is not critically 3-connected and $ u $ is a noncritical vertex of $ G $.

	Suppose $ \delta(G)\geq 4 $. Let $ e=uv $ and $ S(e)=\left\{w_{i}|i=1, 2, \cdots, k(e)\right\} $. By Lemma 3.4 (6), $ C_{u}(e)=\left\{u\right\} $. Since $ d_{G}(u)\geq 4 $, then $ k(e)\geq3 $.

{\bf Case 1.} $ k(e)=3 $. 
\par
As Lemma 3.4 (1) and $ \kappa(G)=k(e)=3 $, then $ N_{G}(v)=\left\{u\right\}\cup S(e)$ and $ S(e)\subseteq N_{S(e)}\big(u, v, D(e)\big) $. By Lemma 3.4 (5), $ N_{G}(w_{1})\subseteq \left\{v, t_{1}\right\}\cup N_{G}(v) $, where $ t_{1}\in N_{D(e)}(w_{1}) $. Since $ d_{G}(w_{1})\geq 4 $, then either $ w_{1}w_{2}\in E(G) $ or $ w_{1}w_{3}\in E(G) $. Without loss of generality, assume $ e_{1}=w_{1}w_{2}\in E(G) $. Clearly, $ \left\{u, v\right\}\subseteq S(e_{1})$. By Lemma 3.4 (4), $ N_{ D(e_{1})}(v)\neq \emptyset$ and then $ w_{3}\in D(e_{1})$. Hence $ \left\{w_{1}, w_{2}\right\}\cap N_{G}(w_{3})=\emptyset $. By Lemma 3.4 (5), $ N_{G}(w_{3})\subseteq \left\{v, t_{2}\right\}\cup N_{G}(v) $, where $ t_{2}\in N_{D(e)}(w_{3}) $. Since  $ \left\{w_{1}, w_{2}\right\}\cap N_{G}(w_{3})=\emptyset$, then $ d_{G}(w_{3})\leq 3 $, a contradiction.  
	
	{\bf Case 2.} $ k(e)\geq 4 $.
	
	 By Lemma 3.5, $ C(e)=\left\{u, v\right\} $ and assume $ N_{S(e)}(u, v)=\left\{w_{1}, w_{2}\right\} $. 
	
	{\bf Case 2.1.} $ k(e)=4 $.
	
	As $ \kappa(G)=3 $, assume $ w_{1}\in N_{S(e)}\big(D(e)\big) $. By Lemma 3.4 (5), $ d_{G}(v)=4 $ and $ N_{G}(w_{1})\subseteq \left\{v, t_{3}\right\}\cup N_{G}(v) $, where $ t_{3}\in N_{D(e)}(w_{1}) $. Since $ N_{S(e)}(u, v)=\left\{w_{1}, w_{2}\right\} $ and $ \delta(G)\geq 4 $, we can assume $ N_{G}(u)=\left\{v, w_{1}, w_{2}, w_{3}\right\} $ and $ N_{G}(v)=\left\{u, w_{1}, w_{2}, w_{4}\right\} $. As $ d_{G}(w_{1})\geq 4 $, then either $ w_{1}w_{2}\in E(G) $ or $ w_{1}w_{4}\in E(G) $.
	\par
   Suppose $ w_{1}w_{2}\in E(G) $. If $ w_{2}w_{3}\notin E(G) $, then $ \left\{v, w_{1}, w_{2}\right\}\cap N_{G}(w_{3})=\emptyset $. This means $ uw_{3} $ is not contained in any triangle, which contradicts Lemma 3.2. If $ w_{2}w_{3}\in E(G) $, assume $ e_{2}=w_{2}w_{3} $. Clearly, $ u\in S(e_{2}) $. By Lemma 3.4 (\ref{result6}), $ C(e_{2})=\left\{w_{2}, w_{3}\right\} $ and $ k(e_{2})=4 $. Then $ \left\{v, u, w_{1}\right\}\subseteq S(e_{2}) $ and $ N_{G}(w_{3})\subseteq \left\{w_{2}\right\}\cup S(e_{2}) $. Since $ v, w_{1}\notin N_{G}(w_{3}) $, then $ d_{G}(w_{3})\leq 3 $, a contradiction. 
  
   \par
  Suppose $ w_{1}w_{4}\in E(G) $.  Let $ e_{3}=uw_{1} $. Clearly, $ v\in S(e_{3}) $. Since $ N_{G}(v)\subseteq N_{G}(u)\cup N_{G}(w_{1}) $, then $ v $ is not adjacent to $ D(e_{3}) $. By Lemma 3.4 (\ref{result5}), $ N_{S(e_{3})}\big(D(e_{3})\big)=S(e_{3})\backslash\left\{v\right\} $. Moreover, since $ C_{u}(e_{3})=\left\{u\right\} $ by Lemma 3.4 (\ref{result7}), then $ \left\{w_{2}, w_{3}\right\}\subseteq S(e_{3}) $. Thus $ \left\{w_{2}, w_{3}\right\}\subseteq N_{S(e_{3})}\big(u, D(e_{3})\big) $. Since $ G $ is claw-free, then $\left\{w_{2}, w_{3}\right\}\cap N_{S(e_{3})}\big(C_{w_{1}}(e_{3})\backslash\left\{w_{1}\right\}\big)=\emptyset$. And as $ \left\{w_{2}, w_{3}\right\}\cap N_{G}(w_{1})=\emptyset $, then $\left\{w_{2}, w_{3}\right\}\cap N_{S(e_{3})}\big(C_{w_{1}}(e_{3})\big)=\emptyset$. Hence $ N_{G}\big(C_{w_{1}}(e_{3})\big)\subseteq\left\{u\right\}\cup \big(S(e_{3})\backslash\left\{w_{2}, w_{3}\right\}\big) $ is a cut set of $ G $. However, since $ k(e_{3})=4 $ by Lemma 3.4 (3), then $ \Big|N_{G}\big(C_{w_{1}}(e_{3})\big)\Big|\leq 3 $, a contradiction.
   
   Thus we conclude $  w_{1}w_{2}\notin E(G) $ and $ w_{1}w_{4}\notin E(G) $, a contradiction.
	\par 	
	{\bf Case 2.2.} $ k(e)\geq 5 $. 
	\par
	By Lemma 3.4 (2), we can assume $ N_{S(e)}\big(C(e)\big)=\left\{w_{1}, w_{2}, w_{3}, w_{4}, w_{5}\right\} $ and $ w_{1}\in N_{S(e)}\big(D(e)\big) $.
	As $ w_{1}\in N_{S(e)}\big(u, v, D(e)\big) $, then $ d_{G}(v)=4 $ by Lemma 3.4 (5). Assume $ N_{G}(v)=\left\{u, w_{1}, w_{2}, w_{5}\right\} $. Since $ N_{S(e)}(u, v)=\left\{w_{1}, w_{2}\right\} $, then $N_{G}(u)=\left\{v, w_{1}, w_{2}, w_{3}, w_{4}\right\} $. 
Let $ e_{4}=uw_{3} $. By Lemma 3.4 (6), $ C_{u}(e_{4})=\left\{u\right\} $. Then $ N_{G}(u)\backslash\left\{w_{3}\right\}\subseteq S(e_{4}) $. Clearly, $ k(e_{4})\geq 4 $.

 If $ k(e_{4})=4 $, then $ S(e_{4})=N_{G}(u)\backslash\left\{w_{1}\right\}$. By Lemma 3.4 (2), $ \emptyset\neq N_{D(e)}(w_{3})\subseteq C_{w_{3}}(e_{4})$. As $ \kappa(G)=3 $, then $ \Big|N_{S(e_{4})}\big(C_{w_{3}}(e_{4})\backslash\left\{w_{3}\right\}\big)\Big|\geq 2 $. Since $ N_{S(e_{4})}\big(C_{w_{3}}(e_{4})\backslash\left\{w_{3}\right\}\big)\subseteq N_{S(e_{4})}(u) $ and $ G $ is claw-free, then $ N_{S(e_{4})}\big(C_{w_{3}}(e_{4})\backslash\left\{w_{3}\right\}, D(e_{4})\big)=\emptyset $. Hence $ \Big|N_{S(e_{4})}\big(D(e_{4})\big)\Big|\leq 2 $, a contradiction.	
 
 If $ k(e_{4})\geq 5 $, then $ \left\{v, w_{1}, w_{2}\right\}\subseteq N_{S(e_{4})}\big(u, D(e_{4})\big) $ by Lemma 3.4 (\ref{result4}). As $ G $ is claw-free, then $ \left\{v, w_{1}, w_{2}\right\}\cap N_{S(e_{4})}\big(C_{w_{3}}(e_{4})\backslash\left\{w_{3}\right\}\big)=\emptyset $. Moreover, by Lemma 3.4 (2), $ \left\{w_{1}, w_{2}\right\}\subseteq N_{S(e)}(u, v, D(e)) $. Then $ w_{3}\notin N_{G}(w_{1})\cup N_{G}(w_{2}) $ by Lemma 3.4 (5). So $\left\{v, w_{1}, w_{2}\right\}\cap N_{S(e_{4})}\big(C_{w_{3}}(e_{4})\big)\\=\emptyset $. Hence $ N_{G}\big(C_{w_{3}}(e_{4})\big)\subseteq\left\{u\right\} \cup N_{S(e_{4})}\big(C(e_{4})\big)\backslash\left\{v, w_{1}, w_{2}\right\}$ is a cut set of $ G $. However, as $ \Big|N_{S(e_{4})}\big(C(e_{4})\big)\Big|=5 $ by Lemma 3.4 (2), then $ \big|N_{G}\big(C_{w_{3}}(e_{4})\big)\big|\leq 3 $, a contradiction. 
			
Therefore we can obtain $ \delta(G)=3 $. \qed

\end{document}